\documentclass[10pt,a4paper]{amsart}[2000]
\usepackage{geometry}
\usepackage{comment}
\usepackage{amsmath}
\usepackage{amssymb}
\usepackage{amsthm}
\usepackage{mathrsfs}
\usepackage[hang]{footmisc}

\usepackage[linktocpage]{hyperref}
\usepackage{tikz, tikz-cd}
\usepackage{color}
\usepackage{mathtools}
\usepackage{enumitem}
\usepackage{stackengine}

\usepackage{graphicx}
\usepackage{esvect}
\makeatletter
\DeclareRobustCommand{\cev}[1]{%
  {\mathpalette\do@cev{#1}}%
}
\newcommand{\do@cev}[2]{%
  \vbox{\offinterlineskip
    \sbox\z@{$\m@th#1 x$}%
    \ialign{##\cr
      \hidewidth\reflectbox{$\m@th#1\vec{}\mkern4mu$}\hidewidth\cr
      \noalign{\kern-\ht\z@}
      $\m@th#1#2$\cr
    }%
  }%
}
\makeatother

\thanks{}

\allowdisplaybreaks

\theoremstyle{plain}

\newtheorem{Thm}{Theorem}[section]

\theoremstyle{definition}

\theoremstyle{plain}
\newtheorem{thm}[Thm]{Theorem}

\newtheorem{cor}[Thm]{Corollary}
\newtheorem{prop}[Thm]{Proposition}

\theoremstyle{definition}
\newtheorem{defn}[Thm]{Definition}

\newtheorem{eg}[Thm]{Example}
\newtheorem{rmk}[Thm]{Remark}

\newtheorem{innercustomthm}{Theorem}
\newenvironment{customthm}[1]
{\renewcommand\theinnercustomthm{#1}\innercustomthm}
{\endinnercustomthm}

\newtheorem{innercustomcor}{Corollary}
\newenvironment{customcor}[1]
{\renewcommand\theinnercustomcor{#1}\innercustomcor}
{\endinnercustomcor}

\newcommand{\B}{B}
\newcommand{\A}{A}
\newcommand{\J}{J}
\newcommand{\K}{\mathcal{K}}
\newcommand{\D}{D}
\newcommand{\Ch}{D}
\newcommand{\Zh}{\mathcal{Z}}
\newcommand{\E}{E}
\newcommand{\Oh}{\mathcal{O}}

\newcommand{\T}{{\mathbb T}}

\newcommand{\N}{{\mathbb N}}
\newcommand{\Z}{{\mathbb Z}}
\newcommand{\C}{{\mathbb C}}
\newcommand{\Q}{{\mathbb Q}}

\newcommand{\aut}{\mathrm{Aut}}
\newcommand{\supp}{\mathrm{supp}}

\newcommand{\eps}{\varepsilon}
\numberwithin{equation}{section}

\newcommand{\id}{\mathrm{id}}

\newcommand{\halpha}{\widehat{\alpha}}
\newcommand{\calpha}{\widehat{\alpha}}

\newcommand{\tih}{\widetilde {h}}

\newcommand\set[1]{\left\{#1\right\}}  
\newcommand\mset[1]{\left\{\!\!\left\{#1\right\}\!\!\right\}}

\newcommand{\CA}[0]{\mathcal{A}} \newcommand{\CB}[0]{\mathcal{B}}
\newcommand{\CC}[0]{\mathcal{C}} \newcommand{\CD}[0]{\mathcal{D}}
 
\newcommand{\CG}[0]{\mathcal{G}} \newcommand{\CH}[0]{\mathcal{H}}

 \newcommand{\CP}[0]{\mathcal{P}}
\newcommand{\CQ}[0]{\mathcal{Q}} \newcommand{\CR}[0]{\mathcal{R}}
 \newcommand{\CT}[0]{\mathcal{T}}
\newcommand{\CU}[0]{\mathcal{U}} 
 
 \newcommand{\CZ}[0]{\mathcal{Z}}

\newcommand{\Ra}[0]{\Rightarrow}
\newcommand{\La}[0]{\Leftarrow}
\newcommand{\LRa}[0]{\Leftrightarrow}

\newcommand{\quer}[0]{\overline}
\newcommand{\eins}[0]{\mathbf{1}}			
\newcommand{\diag}[0]{\operatorname{diag}}
\newcommand{\ad}[0]{\operatorname{Ad}}
\newcommand{\ev}[0]{\operatorname{ev}}
\newcommand{\fin}[0]{{\subset\!\!\!\subset}}
\newcommand{\diam}[0]{\operatorname{diam}}
\newcommand{\Hom}[0]{\operatorname{Hom}}
\newcommand{\dst}[0]{\displaystyle}
\newcommand{\spp}[0]{\operatorname{supp}}
\newcommand{\lsc}[0]{\operatorname{Lsc}}
\newcommand{\del}[0]{\partial}
\newcommand{\GU}[0]{\CG^{(0)}}

\newtheorem{theorem}[Thm]{Theorem}
\newtheorem{lemma}[Thm]{Lemma}

\theoremstyle{definition}

\numberwithin{equation}{Thm}

\title[]{Groupoid strict comparison for Ample Groupoids:\\
Fiberwise Supramenability and Topological Amenability}	

\begin{document}
\global\long\def\floorstar#1{\lfloor#1\rfloor}
\global\long\def\ceilstar#1{\lceil#1\rceil}	

\global\long\def\B{B}
\global\long\def\A{A}
\global\long\def\J{J}
\global\long\def\K{\mathcal{K}}
\global\long\def\D{D}
\global\long\def\Ch{D}
\global\long\def\Zh{\mathcal{Z}}
\global\long\def\E{E}
\global\long\def\Oh{\mathcal{O}}

\global\long\def\T{{\mathbb{T}}}
\global\long\def\BR{{\mathbb{R}}}
\global\long\def\N{{\mathbb{N}}}
\global\long\def\Z{{\mathbb{Z}}}
\global\long\def\C{{\mathbb{C}}}
\global\long\def\Q{{\mathbb{Q}}}

\global\long\def\aut{\mathrm{Aut}}
\global\long\def\supp{\mathrm{supp}}

\global\long\def\eps{\varepsilon}

\newcommand{\dom}{\operatorname{dom}}
\newcommand{\ran}{\operatorname{ran}}

\newcommand{\Prob}{\operatorname{Prob}}

\newcommand{\clop}{\operatorname{Clop}}
\newcommand{\st}{\mathrm{st}}

\newcommand{\UFS}{\textup{(UFS)}}
\newcommand{\USG}{\textup{(USG)}}
\newcommand{\card}[1]{\lvert#1\rvert}

\global\long\def\id{\mathrm{id}}

\global\long\def\halpha{\widehat{\alpha}}
\global\long\def\calpha{\widehat{\alpha}}

\global\long\def\tih{\widetilde{h}}

\global\long\def\opFol{\operatorname{F{\o}l}}

\global\long\def\opRange{\operatorname{Range}}

\global\long\def\opIso{\operatorname{Iso}}
\global\long\def\opisom{\operatorname{Isom}}
\global\long\def\dimnuc{\dim_{\operatorname{nuc}}}

\global\long\def\set#1{\left\{  #1\right\}  }


\global\long\def\mset#1{\left\{  \!\!\left\{  #1\right\}  \!\!\right\}  }

\global\long\def\Ra{\Rightarrow}
\global\long\def\La{\Leftarrow}
\global\long\def\LRa{\Leftrightarrow}

\global\long\def\quer{\overline{}}
\global\long\def\eins{\mathbf{1}}
\global\long\def\diag{\operatorname{diag}}
\global\long\def\ad{\operatorname{Ad}}
\global\long\def\ev{\operatorname{ev}}
\global\long\def\fin{{\subset\!\!\!\subset}}
\global\long\def\diam{\operatorname{diam}}
\global\long\def\Hom{\operatorname{Hom}}
\global\long\def\dst{{\displaystyle }}
\global\long\def\spp{\operatorname{supp}}
\global\long\def\spo{\operatorname{supp}_{o}}
\global\long\def\del{\partial}
\global\long\def\lsc{\operatorname{Lsc}}
\global\long\def\GU{\CG^{(0)}}
\global\long\def\HU{\CH^{(0)}}
\global\long\def\AU{\CA^{(0)}}
\global\long\def\BU{\CB^{(0)}}
\global\long\def\CUU{\CC^{(0)}}
\global\long\def\DU{\CD^{(0)}}
\global\long\def\QU{\CQ^{(0)}}
\global\long\def\TU{\CT^{(0)}}
\global\long\def\CUUU{\CC'{}^{(0)}}
\global\long\def\dom{\operatorname{dom}}
\global\long\def\ran{\operatorname{ran}}
\global\long\def\AUl{(\CA^{l})^{(0)}}
\global\long\def\BUl{(B^{l})^{(0)}}
\global\long\def\HUp{(\CH^{p})^{(0)}}
\global\long\def\sym{\operatorname{Sym}}
\global\long\def\stab{\operatorname{Stab}}
\newcommand{\cat}[0]{\operatorname{CAT}(0)}
\global\long\def\properlength{proper}
\global\long\def\deg{\operatorname{deg}}
\global\long\def\isom{\operatorname{Isom}}
\global\long\def\interior#1{#1^{\operatorname{o}}}
	\global\long\def\ln{\operatorname{ln}}

\author{Chunlin Liu}
	\address{C. Liu: School of Mathematical Sciences, Dalian University of Technology,
Dalian 116024, P.~R. China}
\email{chunlinliu@mail.ustc.edu.cn; chunlin.liu.math@gmail.com}

\author{Xin Ma}
	
	\address{X. Ma:  Institute for Advanced Study in Mathematics, Harbin Institute of Technology, Harbin, China, 150001}
  \email{xma17@hit.edu.cn}


\author{Jianchao Wu}
\address{J. Wu: Shanghai Center for Mathematical Sciences, Fudan University,
Shanghai, 200438, China}

\email{jianchao\_wu@fudan.edu.cn}


\date{\today}

\begin{abstract}
In this paper, we first introduce fiberwise supramenability for locally compact Hausdorff \'{e}tale groupoids with compact unit space. Then, for such a minimal ample groupoid $\CG$, we show that if $\CG$ is fiberwise supramenable or topologically amenable, then the clopen type semigroup $S(\CG)$ is almost unperforated. Consequently, if $\CG$ is $\sigma$-compact, then it has groupoid strict comparison. We then present several applications, including Matui's AH conjecture and pure infiniteness for groupoids and groupoid $C^*$-algebras.
\end{abstract}
\maketitle

\section{Introduction}
The type semigroup of an ample groupoid $\CG$ records finite equidecompositions of
compact open sets in the finite amplification. Its almost unperforation
asserts that
\begin{equation}\label{eq:au}
 (n+1)a\leq nb\quad\Longrightarrow\quad a\leq b
 \qquad(a,b\in S(\CG),\ n\geq1).
\end{equation}
For minimal second-countable ample groupoids, this algebraic property is
equivalent to  dynamical comparison, as shown by Ara--B\"onicke--Bosa--Li
\cite[Theorem A]{ABBL}. Recently, this comparison property has played an important
role in several areas of locally compact Hausdorff \'{e}tale groupoids with
compact unit spaces, such as (tracial) $\CZ$-stability and pure infiniteness
for reduced groupoid $C^*$-algebras (see, e.g., \cite{MW2} and \cite{Mpure}),
and in homology theory, such as Matui's AH conjecture for ample groupoids
(see, e.g., \cite{Matui} and \cite{Li}).

For transformation groupoids, i.e. group actions, Glasner--Liu \cite{GL} recently established comparison for actions of countable amenable groups on compact zero-dimensional spaces. In \cite{GGKN}, Gardella--Geffen--Kranz--Naryshkin showed that every minimal, topologically free, topologically amenable action of an acylindrically hyperbolic group on a compact metrizable space has comparison. For further examples of groups from Bass--Serre theory for which all such actions have comparison, we refer to \cite{MWY}.

In the general groupoid setting, Austad--B\"{o}nicke \cite{AB} showed that if $\CG$ admits a nice geometric length function satisfying a certain uniform doubling condition, then $\CG$ has comparison. Moreover, Geffen--Hua--Kranz \cite{GHK} established comparison for certain Deaconu--Renault groupoids.

As for counterexamples, Boldrini--Prasad \cite{BP} proved that there exists a minimal topologically free action of a free group that has no comparison. However, their example is not topologically amenable. More recently, Keuppens--Martinez \cite{KM} constructed a topologically amenable, minimal, topologically principal groupoid model $\CG$ for the celebrated R\o rdam finite-infinite $C^*$-algebras. By \cite[Corollary B]{KM}, such a groupoid cannot have comparison. Nevertheless, this groupoid is not ample and actually its unit space has infinite covering dimension (see the statement immediately before \cite[Corollary B]{KM}).

In this paper, motivated by \cite{GL}, we prove that the type semigroup of a
groupoid is almost unperforated whenever the groupoid satisfies a certain
amenability condition. Two distinct notions of amenability are relevant here.
The first is the usual topological amenability, for which \cite{ADR} serves as
a standard reference. The second, called (ubiquitous) \textit{fiberwise
amenability}, was introduced in \cite{MW} as a geometric amenability condition
describing the coarse structure of the underlying groupoid and it corresponds to
the existence of an invariant probability measure on the unit space. For a more
detailed comparison between these two notions, we refer to the introduction of
\cite{MW}. For a transformation groupoid
$\CG=X\rtimes_\alpha \Gamma$ associated with an action
$\alpha:\Gamma\curvearrowright X$, the topological amenability of $\CG$ is
exactly the topological amenability of $\alpha$, while the fiberwise
amenability of $\CG$ corresponds to the amenability of the acting group
$\Gamma$. In general, however, neither notion necessarily implies the other for
groupoids; see \cite[Example 5.23]{MW}.

In addition, motivated by supramenable groups, we introduce \emph{fiberwise supramenability} for groupoids. Let $\CG$ be a locally compact Hausdorff \'{e}tale groupoid with compact unit space. Fiberwise supramenability asks that, for every unit $u$, every prescribed compact control set $K$, and every $\varepsilon>0$, every nonempty $E\subseteq \CG_u=s^{-1}(u)$ contains a nonempty finite $F$ such that
\begin{equation}\label{eq:intro-folner}
 \card{KF\cap E}\leq(1+\varepsilon)\card F.
\end{equation}
The precise definition appears in Section~\ref{sec:folner}. The terminology reflects the established distinction between amenability and supramenability for groups and coarse spaces and we refer to Kellerhals--Monod--R\o rdam \cite{KMR} for an example. It was proved in \cite{AB} that if a groupoid $\CG$ has (fiberwise) subexponential growth, then $\CG$ is topologically amenable. In Section~\ref{sec:folner}, we will show that (fiberwise) subexponential growth implies that $\CG$ is fiberwise supramenable. However, we do not know whether fiberwise supramenability implies topological amenability as well.

The following is our first main result. We refer to Section~\ref{sec:preliminaries} for the definitions of \textit{groupoid strict comparison}. We note that groupoid strict comparison is also referred to as \textit{dynamical comparison}, or simply comparison, in the literature. However, we prefer to keep our notation as in \cite{MW2}, because groupoid strict comparison is equivalent to dynamical comparison in this ample setting, but is strictly weaker in general, though still sufficient for applications to $C^*$-algebras (see \cite[Remark 7.4]{MW2}).

\begin{customthm}{A}\label{thm:A}
Let $\CG$ be a locally compact Hausdorff \'{e}tale ample groupoid with compact
unit space. If $\CG$ is fiberwise supramenable, then $S(\CG)$ is almost unperforated.  Then if $\CG$ is $\sigma$-compact and minimal, then $\CG$ has groupoid strict comparison
\end{customthm}

Our second result uses topological amenability instead of a condition on
subsets of fibers.

\begin{customthm}{B}\label{thm:B}
   Let $\CG$ be a locally compact Hausdorff topologically amenable 
ample \'{e}tale groupoid with compact unit space. Then
\[
 3a\leq2b\quad\Longrightarrow\quad a\leq b
 \qquad(a,b\in S(\CG)).
\]
If $\CG$ is minimal, then $S(\CG)$ is almost unperforated. Then if further $\CG$ is $\sigma$-compact, then $\CG$ has groupoid strict comparison.
\end{customthm}

Applying to transformation groupoids, Theorem \ref{thm:B}  yields dynamical comparison for all topologically amenable
minimal actions of countable groups on compact metrizable zero-dimensional
spaces. This proves the zero-dimensional cases of Conjecture D in
\cite{GGKN} and Problem~XXXIII in  Schafhauser-Tikuisis-White's survey \cite{STW}. 

The following are applications of Theorems \ref{thm:A} and \ref{thm:B}. The first one is to the Matui's AH conjecture. It was proved in \cite[Corollary E]{Li} by Li that if a locally compact Hausdorff \'{e}tale ample groupoid $\CG$ is minimal and has groupoid strict comparison, then $\CG$ satisfies Matui's AH conjecture. 

\begin{customcor}{C}[Theorem \ref{thm: AH}]
    Let $\CG$ be an ample minimal groupoid on an infinite compact unit space such that $\CG$ is fiberwise supramenable or $\CG$ is topological amenable. Then $\CG$ satisfies the AH conjecture.
\end{customcor}

We also have the following applications to pure infiniteness of groupoids and groupoid $C^*$-algebras.

\begin{customcor}{D}[Theorem \ref{thm: pure inf}]
    Let $\CG$ be an ample minimal topological amenable groupoid on an infinite compact unit space. 
    \begin{enumerate}
        \item If there is no $\CG$-invariant probability measure on $\GU$, then $\CG$ is purely infinite.  
        \item Suppose $\CG$ is topological principal. Then its reduced groupoid $C^*$-algebra $C^*_r(\CG)$ is unital, simple,  nuclear and purely infinite $C^*$-algebra. 
        \item If $\CG$ is further assumed to be second countable, then $C^*_r(\CG)$ is a unital Kirchberg algebra satisfying the UCT.
    \end{enumerate}
     
\end{customcor}
In particular, any minimal, topologically free, topologically amenable action of non-amenable groups on zero-dimensional spaces will provide unital Kirchberg algebras, that is classfied by Elliott's invariant. This, in particular, has addressed \cite[Problem~XXXII]{STW} in this case. In addition, we also record an application to almost elementariness in Theorem \ref{thm:AE}.

\begin{rmk}
We would like to mention that, during the preparation of this paper,
Elek--Tim\'{a}r posted a preprint on arXiv \cite{ET} containing closely
related results (\cite[Theorem 1]{ET}) to our Theorem \ref{thm:B}. Their results studied topologically amenable groupoids under
the additional assumptions of principality and second countability.
\end{rmk}

\section{Preliminaries\label{sec:preliminaries}}
\subsection{Backgrounds on groupoids}
We recall some basic backgrounds on
\'{e}tale groupoids and $C^{*}$-algebras.
We refer to \cite{Renault} and \cite{Sims}
as references for groupoids and we record several fundamental definitions
and results for locally compact Hausdorff \'{e}tale groupoids here. 
\begin{defn}
A \textit{groupoid} $\CG$ is a set equipped with a distinguished
subset $\CG^{(2)}\subseteq\CG\times\CG$, called the set of \textit{composable
pairs}, a product map $\CG^{(2)}\rightarrow\CG$, denoted by $(\gamma,\eta)\mapsto\gamma\eta$
and an inverse map $\CG\rightarrow\CG$, denoted by $\gamma\mapsto\gamma^{-1}$
such that the following hold 
\begin{enumerate}[label=(\roman*)]
\item If $(\alpha,\beta)\in\CG^{(2)}$ and $(\beta,\gamma)\in\CG^{(2)}$
then so are $(\alpha\beta,\gamma)$ and $(\alpha,\beta\gamma)$. In
addition, $(\alpha\beta)\gamma=\alpha(\beta\gamma)$ holds in $\CG$.
\item For all $\alpha\in\CG$ one has $(\gamma,\gamma^{-1})\in\CG^{(2)}$
and $(\gamma^{-1})^{-1}=\gamma$.
\item For any $(\alpha,\beta)\in\CG^{(2)}$ one has $\alpha^{-1}(\alpha\beta)=\beta$
and $(\alpha\beta)\beta^{-1}=\alpha$. 
\end{enumerate}
Every groupoid is equipped with a subset $\GU=\{\gamma\gamma^{-1}:\gamma\in\CG\}$
of $\CG$. We refer to elements of $\GU$ as \textit{units} and to
$\GU$ itself as the \textit{unit space}. We define two maps $s,r:\CG\rightarrow\GU$
by $s(\gamma)=\gamma^{-1}\gamma$ and $r(\gamma)=\gamma\gamma^{-1}$,
respectively, in which $s$ is called the \textit{source} map and
$r$ is called the \textit{range} map. 
\end{defn}
\label{paragraph:sections-in-groupoids}
When a groupoid $\CG$ is endowed with a locally compact Hausdorff
topology under which the product and inverse maps are continuous,
the groupoid $\CG$ is called a locally compact Hausdorff groupoid.
A locally compact Hausdorff groupoid $\CG$ is called \textit{\'{e}tale}
if the range map $r$ is a local homeomorphism from $\CG$ to itself,
which means for any $\gamma\in\CG$ there is an open neighborhood
$U$ of $\gamma$ such that $r(U)$ is open and $r|_{U}$ is a homeomorphism.
In this case, since the map of taking inverses is an involutive homeomorphism
on $\CG$ that intertwines $r$ and $s$, thus the source map $s$
is also a local homeomorphism. 
A subset $S$ in $\CG$ is called an \emph{$s$-section} (respectively, an \emph{$r$-section})
if there is an open neighborhood $U$ of $S$ in $\CG$ such that
the source map $s$ (respectively,
the range map $r$) restricts to a homeomorphism from $U$
onto an open subset of $\GU$. 
It is called a \emph{bisection} if it is both an $s$-section and an $r$-section. 
It is not hard to see a locally compact
Hausdorff groupoid is \'{e}tale if and only if its topology has a basis
consisting of open bisections. We say a locally compact Hausdorff
\'{e}tale groupoid $\CG$ is \textit{ample} if its topology has a basis
consisting of compact open bisections.

\begin{eg}
\label{exa:transformation-groupoid}
Let $X$ be a locally compact
Hausdorff space and $\Gamma$ be a discrete group. Then any action
$\Gamma\curvearrowright X$ by homeomorphisms induces a locally compact
Hausdorff \'{e}tale groupoid 
\[
X\rtimes\Gamma\coloneqq\{(\gamma x,\gamma,x):\gamma\in\Gamma,x\in X\}
\]
equipped with the relative topology as a subset of $X\times\Gamma\times X$.
In addition, $(\gamma x,\gamma,x)$ and $(\beta y,\beta,y)$ are composable
only if $\beta y=x$ and 
\[
(\gamma x,\gamma,x)(\beta y,\beta,y)=(\gamma\beta y,\gamma\beta y,y).
\]
One also defines $(\gamma x,\gamma,x)^{-1}=(x,\gamma^{-1},\gamma x)$
and announces that $\GU\coloneqq\{(x,e_{\Gamma},x):x\in X\}$. It
is not hard to verify that $s(\gamma x,\gamma,x)=x$ and $r(\gamma x,\gamma,x)=\gamma x$.
The groupoid $X\rtimes\Gamma$ is called a \textit{transformation
groupoid}. 
\end{eg}
The following are several basic properties of locally compact Hausdorff
\'{e}tale groupoids whose proofs could be found in \cite{Sims}.
\begin{prop}
Let $\CG$ be a locally compact Hausdorff \'{e}tale groupoid. Then $\GU$
is a clopen set in $\CG$. 
\end{prop}

For any set $D\subseteq\GU$, Denote by 
\[
\CG_{D}\coloneqq\{\gamma\in\CG:s(\gamma)\in D\},\ \CG^{D}\coloneqq\{\gamma\in\CG:r(\gamma)\in D\},\ \text{and}\ \ \CG_{D}^{D}\coloneqq\CG^{D}\cap\CG_{D}.
\]
For the singleton case $D=\{u\}$, we write $\CG_{u}$, $\CG^{u}$
and $\CG_{u}^{u}$ instead for simplicity. In this situation, we call
$\CG_{u}$ a \textit{source fiber} and $\CG^{u}$ a \textit{range
fiber}. In addition, each $\CG_{u}^{u}$ is a group, which is called
the \textit{isotropy} at $u$. We also denote by \[\opIso(\CG)=\bigcup_{u\in \GU}\CG^u_u=\{x\in \CG: s(x)=r(x)\}\] the isotropy of the groupoid $\CG$. We say a groupoid $\CG$ is \textit{principal}
if $\opIso(\CG)=\GU$. A groupoid $\CG$ is called \textit{topologically
principal} if the set $\{u\in\GU:\CG_{u}^{u}=\{u\}\}$ is dense in
$\GU$. The groupoid $\CG$ is also said to be \textit{effective} if $\opIso(\CG)^o=\GU$. Recall that effectiveness is equivalent to topological principality  if $\CG$ is second countable (See \cite[Lemma 4.2.3]{Sims}). Therefore, effectiveness is equivalent to the topological freeness of an action of a countable discrete group acting on a compact metrizable space by looking at the corresponding transformation groupoid.

A subset $D$ in $\GU$ is called $\CG$-\textit{invariant}
if $r(\CG D)=D$, which is equivalent to the condition $\CG^{D}=\CG_{D}$.
Note that $\CG|_{D}\coloneqq\CG_{D}^{D}$ is a subgroupoid of $\CG$
with the unit space $D$ if $D$ is a $\CG$-invariant set in $\GU$.
A groupoid $\CG$ is called \textit{minimal} if there are no proper
non-trivial closed $\CG$-invariant subsets in $\GU$.

\begin{defn}
\label{def:invariant-measure}A measure on the unit space of an \'{e}tale groupoid $\CG$ is \emph{invariant} if $\mu(r(U))=\mu(s(U))$
for any Borel measurable bisection $U$. We write $M(\CG)$ for the collection
of all \emph{invariant regular Borel probability measures} on $\CG^{(0)}$. \end{defn}

\begin{defn}(\cite[Definition 7.1]{MW2})
	\label{defn:subequivalence}
	Let $\CG$ be an \'{e}tale groupoid. Let $K$ be a compact subset of $\CG^{(0)}$ and let $U,V \subseteq \CG^{(0)}$. 
	\begin{enumerate}[label=(\roman*)]
		\item We write $K\prec_{\CG}V$ if there is an open $s$-section $S$ such that $K \subseteq r(S)$ and $s(S) \subseteq V$. 
		\item We write $U\precsim_{\CG}V$
		if $K\prec_{\CG}V$ for every compact subset $K\subseteq U$. 
	\end{enumerate}
\end{defn}

\begin{defn}(\cite[Definition 8.1]{MW2})
	\label{defn:comparison}
	We say a minimal \'{e}tale groupoid $\CG$ with a compact unit space has \emph{groupoid strict comparison} (or simply \emph{groupoid
		comparison} or\emph{ comparison}) if, for any nonempty open sets $U,V\subseteq\CG^{(0)}$,
	we have $U\precsim_{\CG}V$ whenever $\mu(U)<\mu(V)$ for all $\mu\in M(\CG)$.
\end{defn}

\begin{rmk}
    Let $\CG$ be a locally compact Hausdorff \'{e}tale ample minimal groupoid with a compact unit space. Then it follows from \cite[Remark 7.4]{MW2} that the groupoid strict comparison is equivalent to the following statement: for any non-empty compact open sets $U, V$ in $\GU$, if $\mu(U)<\mu(V)$ holds for any $\mu\in M(\CG)$, then there exists compact open bisections $O_1,\dots, O_n$ such that 
    $U=\bigsqcup_{i=1}^ns(O_i)$ and $\bigsqcup_{i=1}^n r(O_i)\subset V$. Note that this is the usual dynamical comparison for ample groupoids in the literature.
\end{rmk}

Let $\CG$ be a locally compact Hausdorff \'{e}tale groupoid.  $C_c(\CG)$ be the set of continuous functions from $\CG$ to $\mathbb{R}$ with compact support.  If $\CG$ is an \'etale groupoid, then a \emph{continuous system of probability measures} for $\CG$ is a system $\{\lambda^u\mid u\in \CG^{(0)}\}$ of Radon probability measures $\lambda^u$ on $\CG$ with the support of $\lambda^u$ contained in $\CG^u:=\{g\in \CG \mid r(g) =u\}$ such that for $f\in C_c(\CG)$, the function $u \mapsto \int f d\lambda^u$ is continuous. Note that since $\CG^u$ is discrete, a Radon probability measure $\lambda^u$ on $\CG^u$ amounts to a function $\lambda^u\colon \CG^u \to [0,\infty)$ with $\Sigma_{z\in \CG^{u}} \lambda^u(z) = 1$.

An \emph{approximate invariant continuous mean} for $\CG$ is a net $\lambda_i$ of continuous systems of probability measure for $\CG$ such that the net $(M_i: \gamma \mapsto ||\lambda_i^{r(\gamma)}(\gamma \cdot) - \lambda_i^{s(\gamma)}(\cdot)||_1)_i$ of functions from $\CG$ to $[0,\infty)$ has the property that $M_i\mid_K \to 0$ uniformly for every compact $K\subseteq \CG$.

\begin{defn}
    Let $\CG$ be an \'etale groupoid. We say that $\CG$ is (topologically) amenable if $\CG$ admits an approximate invariant continuous mean.
\end{defn}

We remark that the topological amenability is permanent under finite amplification,
open subgroupoids, and closed invariant reductions. We also make a standing assumption that \textbf{all groupoids in the following are locally compact Hausdorff \'{e}tale groupoids with a compact unit space.}

\subsection{The type semigroup}
Following the vocabulary in \cite{Melleray},
let $\mathcal R=\N\times\N$ be the discrete pair groupoid, with an arrow
$(i,j)$ from $j$ to $i$. Put $\CG^{\st}=\CG\times\mathcal R$.
An element of $S(\CG)$ is represented by a compact open subset of
$\GU\times\N$. Such a set $A$ occupies only finitely many labels, which means $A\cap (\GU\times \{n\})=\emptyset$ for all large enough $n$. Two
representatives are equivalent if a compact open bisection of $\CG^{\st}$ has
one as its source and the other as its range. Addition is disjoint union after
retagging. The algebraic order satisfies
\begin{equation}\label{eq:subequiv}
 [A]\leq[B]\quad\Longleftrightarrow\quad
 \text{a compact open bisection has source }A
 \text{ and range contained in }B.
\end{equation}
This is the usual type semigroup; see \cite{ABBL}. Indeed, an algebraic
inequality $[A]+c=[B]$ restricts an equidecomposition to $A$, and conversely
the unused compact open part of $B$ represents the remainder $c$.

A commutative monoid $S$ is said to be \textit{conical} if for any $x, y\in S$, the equality $x+y=0$ implies that $x=y=0$. In addition, $S$ is said to be a \textit{refinement monoid} if for any $a, b,c,d\in S$ satisfying $a+b=c+d$, then there exists $w, x,y,z\in S$ such that $a=w+x$, $b=x+y$, $c=w+y$, and $d=x+z$. It is straightforward to see that the monoid $S(G)$ is conical and has refinement. An element $u$ in a monoid $S$ is said to be an \textit{order unit} if for any $x\in S$ there exists an $n\in \N$ such that $x\leq nu$. The monoid $S$ is said to be \textit{simple} if every non-zero element is an order unit.

\begin{rmk}\label{rmk: type semigroup conical and simple}
First it is straightforward to see $S(\CG)$ is conical. In the case that $\CG$ is minimal,
every nonzero element in $S(\CG)$ is an order unit and therefore $S(\CG)$ is simple.
To see this, let $[B]\in S(\CG)$ and $U\subset \GU$ be non-empty such that $[U]\leq [B]$. Minimality and compactness of $\GU$ yield finitely many compact open
bisections with sources covering $\GU$ and ranges in $U$. Partition their
sources into disjoint clopen pieces and use separate target labels. This
gives $[\GU]\leq m[U]$ for some $m$. Then note that every $[A]\in S(\CG)$  is bounded by $n[\GU]$ for some $n\in \N$. Therefore, one has $[A]\leq n[\GU]\leq nm[U]\leq nm[B]$.  
\end{rmk}

\subsection{An algebraic comparison principle}

We recall the refinement-monoid result used only for
Theorem~\ref{thm:B}. A simple conical monoid $M$ with order unit $u$ has
\emph{weak comparability} if, for every $c\ne0$, there is $k\geq1$ such that
\[
 kd\leq u\quad\Longrightarrow\quad d\leq c.
\]
For a simple conical refinement monoid, weak comparability implies almost
unperforation. We refer to \cite[Proposition 1.6.8]{Wehrung} and
\cite[Proposition 2.20 and its proof]{Melleray}.

\begin{lemma}\label{lem:three-algebra}
Let $M$ be a simple conical refinement monoid. If $3a\leq2b$ implies
$a\leq b$ for all $a,b\in M$, then $M$ is almost unperforated.
\end{lemma}
\begin{proof}
The zero monoid is immediate. Otherwise it suffices to show $M$ has the weak comparability. Fix an order unit $u$ and $c\ne0$ in $M$.
Choose $n\geq1$ with $u\leq2^n c$. Choose $k=3^n$ and suppose $kd=3^n d\leq u$. This implies
$3^n d\leq2^n c$. Apply the hypothesis to
$3(3^{n-1}d)\leq2(2^{n-1}c)$ and repeat to obtain $d\leq c$.
This implies that $M$ has weak comparability and thus $M$ is almost unperforated.
\end{proof}

\section{Fiberwise supramenability}\label{sec:folner}
\textit{Fiberwise amenability} and its refinement \textit{ubiquitous fiberwise amenability}, as new geometric amenability properties for \'etale groupoid, were introduced in \cite{MW}. Such notions are different from the usual topological amenability of groupoid as demonstrated in \cite[Example 5.23]{MW}. We refer to the introduction of \cite{MW} for more information. In addition, they have played important roles in the study of (tracial) $\CZ$-stability of reduced groupoid $C^*$-algebras \cite{MW2}. In this paper, motivated by supreamenable groups, we introduce stronger properties for \'{e}tale groupoids called (uniform) fiberwise supreamenability. 

\begin{defn}\label{def:fs}
We say that $\CG$ is \emph{fiberwise supramenable} if for any $u\in \GU$, $E\subset \CG_u$,  compact set $K$, and $\varepsilon>0$, there exists a finite $F\subset E$ such that 
\[|KF\cap E|\leq (1+\varepsilon)|F|.\]

We say $\CG$ is \emph{uniformly fiberwise supramenable} if, for every such $K$ and
$\varepsilon$, there is a compact $L\subseteq G$ satisfying the condition that
for any $u\in \GU$ and any $E\subset \CG_u$ and a $\gamma\in E$, there exists  and a (finite) $F\subset E\cap L\gamma$ such that 
\[\card{KF\cap E}\leq(1+\varepsilon)\card F.\]
\end{defn}

In general, fiberwise amenability is not equivalent to ubiquitous fiberwise amenability unless we make extra assumptions on $\CG$ such as minimality or being a transformation groupoid. We refer to \cite[Theorem C, Proposition 5.18]{MW} for more information. However, we will show in the following that their respective strengthenings, fiberwise supramenable and uniform fiberwise supramenable, coincide. The proof has a similar flavor with the local slice lemma in \cite[Lemma 5.11]{MW}. But we have managed to avoid using the length function here.

\begin{prop}\label{prop:pointwise and uniform supramen equivalent}
    Let $\CG$ be a groupoid on a compact unit space. Then $\CG$ is fiberwise supramenable if and only if $\CG$ is uniformly fiberwise supramenable.
 \end{prop}
 \begin{proof}
     It only suffices to show fiberwise supramenability implies uniform fiberwise supramenability. Indeed, we denote by $F(\CG)$ the set of all closed set in $\CG$, including the empty set. Equipped with the Fell topology, $F(\CG)$ is a compact Hausdorff space because $\CG$ is locally compact Hausdorff. Define 
     \[\CP=\{(u, E)\in \GU\times F(\CG): u\in E\subset \CG_u\}.\]
     Note that all subsets of $\CG_u$ is closed because $\CG_u$ itself is discrete. In addition, using the Fell topology, it is not hard to verify that $\CP$ is closed and thus compact.

     Now, let $K$ be a compact set in $\CG$ containing $\GU$ and $\varepsilon>0$. Let $(u, E)\in \CP$. By fiberwise supramenability, one may choose a $F=\{f_1,\dots, f_n\}$ such that $|KF\cap E|\leq (1+\varepsilon)n$.  Since $\CG$ is Hausdorff, choose disjoint compact bisection $B_1,\dots, B_n$ with $f_i\in B_i^o\subset B_i$ for all $i=1,\dots, n$. Define a compact $C=\bigcup_{i=1}^n B_i$, which satisfies $C\cap G_u=F$. This further implies that 
     \begin{equation}\label{eq: fiber u}
         E\cap KC= \CG_u\cap E\cap KC= K(C\cap  \CG_u)\cap E= E\cap KF. 
     \end{equation}
      Write $KF\cap E=\{h_1, \dots, h_m\}$ and choose an open bisection $V_j\ni h_j$ for each $j=1,\dots, m$. Define $D=KC\setminus \bigcup_{j=1}^mV_j$, which is compact. Then the equality \eqref{eq: fiber u} implies $E\cap D=\emptyset$.  Then define an open neighborhood 
      \[O=\big(\bigcap_{i=1}^n s(B_i^o)\big) \times \big(\bigcap_{i=1}^n\{E\in F(\CG): E\cap B^o_i\neq \emptyset \}\cap \{E\in F(\CG): E\cap D=\emptyset\}  \big)\]
      for the $(u, E)\in \CP$. For any $(u', E')\in O$, define $F'=E'\cap \bigcup_{i=1}^n B_i$. Note that $F'\subset C$ and $|F'|=n$. Furthermore, note that 
      $KF'\cap E'\subset KC\cap E'\subset \bigcup_{j=1}^mV_j$ and therefore, one has
      \begin{equation}\label{eq2}
          |KF'\cap E'|\leq m\leq (1+\varepsilon)n=(1+\varepsilon)|F'|.
      \end{equation}

Then by the compactness of $\CP$, one may find finitely many such $O_1, \dots, O_N$ to cover $\CP$, which yield corresponding compact sets $C_1, \dots, C_N$. Define $L=\bigcup_{i=1}^NC_i$. Now let $E\subset \CG_u$ for some $u\in \GU$ and a $\gamma\in E$, define $E_0=E\gamma^{-1}\subset \CG_{r(\gamma)}$. Then $r(\gamma)\in E_0$ and therefore, $(r(\gamma), E_0)\in \CP$. Then by what we obtained above, there exists a $F_0\subset L\cap E_0$ satisfying the inequality \eqref{eq2}. Let $F=F_0\gamma$ and thus
\[|KF\cap E|=|KF_0\cap E_0|\leq (1+\varepsilon)|F_0|=(1+\varepsilon)|F|.\]
\end{proof}

Therefore, in the following, we will not distinguish these two notions. For transformation groupoids with compact nonempty unit space, like fiberwise amenability in \cite{MW}, the fiberwise suprameanbility is the
 F\o lner characterization of supramenability of the acting group.
We remark that it is demonstrated in \cite{KMR} that amenable but nonsupramenable groups exist.

In order to deal with the type semigroup, we need a strengthening of fiberwise supramenability that deals with ``multisets''. 
Write $[m]=\set{1,\ldots,m}$ and let
$\pi:\CG_u\times[m]\to \CG_u$ be the projection. If
$\emptyset\ne E\subseteq \CG_u\times[m]$, define, for finite $F\subseteq E$, a set
\begin{equation}\label{eq:label-neighborhood}
 N_K^E(F)=E\cap\bigl(K\pi(F)\times[m]\bigr).
\end{equation}

\begin{lemma}\label{lem:labels}
Suppose $\CG$ is fiberwise supramenable. For every $u\in \GU$, $m\geq1$,
$\emptyset\ne E\subseteq \CG_u\times[m]$, compact symmetric $K\supseteq \GU$,
and $\varepsilon>0$, there is a nonempty finite $F\subseteq E$ such that
\[
 \card{N_K^E(F)}\leq(1+\varepsilon)\card F.
\]
\end{lemma}
\begin{proof}
Suppose this fails. Set $\lambda=1+\varepsilon>1$. Then every nonempty finite
$F\subseteq E$ satisfies $\card{N_K^E(F)}>\lambda\card F$.
Define $D=\pi(E)$. For a nonempty finite $A\subseteq D$, define
\[
 F_0=E\cap(A\times[m]),\qquad F_{j+1}=N_K^E(F_j).
\]
These are finite, nonempty subsets of $E$ such that $|F_0|\geq |A|$ and
$\pi(F_j)\subseteq K^jA\cap D$. Hence, for every $j\geq1$, one has
\[
 \lambda^j\card A<\card{F_j}\leq m\card{K^jA\cap D}.
\]
Choose $j$ so that $\lambda^j/m>1$ and then choose
$0<\delta<\lambda^j/m-1$. Note that the displayed inequality is valid for every such
$A$. This contradicts the fiberwise supramenability for $D$, the compact set $K^j$, and the
$\delta$.
\end{proof}

\begin{rmk}\label{rem:permanence}
We remark that the property of fiberwise supramenability passes to open subgroupoids with compact unit spaces and to $\CG_D$ for a
 compact open subset $D$ of $\GU$. In addition, fiberwise supramenability also passes
to finite amplifications $\CG\times\mathcal R_m$. To show this, first note that every compact set in $\CG\times \CR_m$ is contained in $K\times \CR_m$ for a compact $K\subset \CG$. Let $E\subset \CG_u\times [m]$ and $F\subset E$, observe that $(K\times \CR_m)\cdot F\subset N_K^E(F)$ by the definition \eqref{eq:label-neighborhood}. Now, the permanence result follows from  Lemma~\ref{lem:labels}.
\end{rmk}

We then show that the (uniform) local subexponential growth of groupoid $\CG$ will imply that $\CG$ is (uniformly) fiberwise supramenable.

\begin{defn}\label{def:growth}
For compact symmetric $K\supseteq \GU$, set
$\beta_K(j)=\sup_{u\in \GU}\card{K^j\cap G_u}$.
We say that $\CG$ has \emph{local subexponential growth},  if
\begin{equation}\label{eq:USG}
 \limsup_{j\to\infty}\beta_K(j)^{1/j}=1
 \quad\text{for every compact symmetric }K\supseteq \GU.
\end{equation}
\end{defn}

The word ``local'' refers to allowing every compact control instead of fixing
a compact generating set (if exists) for the whole groupoid.

\begin{rmk}\label{rmk: pointwise subexponential}
One may define \emph{pointwise} version for subexponential growth to be

\[\limsup_{j\to \infty}\card{K^j\cap G_u}^{1/j}=1\]
for every compact symmetric $K\supset \GU$ and every $u\in \GU$. In the similar manner as in Proposition \ref{prop:pointwise and uniform supramen equivalent}, one can actually show this pointwise version is  equivalent to Definition \ref{def:growth}. Therefore, we will not differentiate the notions.
\end{rmk}

\begin{prop}\label{prop:growth}
Let $\CG$ be a groupoid. If $\CG$ is of local subexponential growth, then $\CG$ is fiberwise supramenable.
\end{prop}
\begin{proof}
Fix a compact $K\supset \GU$ and $\varepsilon>0$. In the uniform case choose $N\geq1$ with
\[
 \beta_K(N)<(1+\varepsilon)^N.
\]
Given $\emptyset\ne E\subseteq G_u$, choose $\gamma\in E$ and set
$F_j=E\cap K^j\gamma$ for $0\leq j\leq N$. Then
$F_0=\set\gamma$ and
\[K\cdot F_j\cap E=K\cdot(E\cap K^j\gamma)\cap E\subset K\cdot K^j\gamma\cap E=F_{j+1}.\]
Suppose none of $F_j$ for $0\leq j\leq N-1$ satisfies 
\[|KF_j\cap E|\leq (1+\varepsilon)|F_j|.\]
Then necessarily one has $|F_{j+1}|\geq |KF_j\cap E|>(1+\varepsilon)|F_j|$. Therefore, by induction, one has $|F_N|>(1+\eps)^N$. 

But right multiplication by $\gamma$
is a bijection from $K^N\cap G_{r(\gamma)}$ to $K^N\gamma$. This implies that
\[
 \card{F_N}\leq\card{K^N\cap G_{r(\gamma)}}\leq\beta_K(N)<(1+\eps)^N,
\]
which is a contradiction. Since $K^j\subseteq K^N$ for $j\leq N$, the compact
set $L=K^N$ witnesses that $\CG$ is uniformly fiberwise supramenable.
\end{proof}

Using the language of length functions on groupois, we have the following.  Roughly speaking, a length function $\ell: \CG\to [0, \infty)$ is a function that vanishes on units; is invariant under inversion; and
is subadditive on composable pairs. We refer to, e.g., \cite{AB} and \cite{MW} for more information on coarse length functions. We remark that the uniformly fiberwise coarse length function used in \cite[Definition 2.4]{AB} satisfies the assumption of the following corollary.

\begin{cor}\label{cor:length}
Suppose $\CG$ admits a length function $\ell:G\to[0,\infty)$ that is controlled (i.e., bounded
on compact sets) and
\[
 \gamma_\ell(t)=\sup_{u\in \GU}\card{\set{g\in \CG_u:\ell(g)\leq t}}
\]
is finite for every $t$. Suppose 
$\log\gamma_\ell(t)/t\to0$ as $t\to\infty$. Then $\CG$ is of local subexponential growth and thus fiberwise supramenable.
\end{cor}
\begin{proof}
Given a symmetric compact $K\supset \GU$, choose $c\geq1$ with $\ell\leq c$ on $K$.
Subadditivity of the length function implies $\beta_K(j)\leq\gamma_\ell(cj)$. Therefore, $\CG$ is of local subexponential growth and thus fiberwise supramenable by  Proposition~\ref{prop:growth}.
\end{proof}

\section{Clopen matchings and controlled expansion}\label{sec:matching}
In this section, we apply the techniques on clopen matchings developed in \cite{GL} to ample groupoids. 
\subsection{Finite clopen presentations}
A bipartite multigraph consisted of two disjoint vertex sets $L$ and $R$, an edge set $E$ as well as maps $s: E\to L$ and $r:E\to R$. In our setting, $L$ and $R$ are compact zero-dimensional Hausdorff spaces. A
\emph{finite clopen presentation} of a bipartite multigraph consists of
finitely many homeomorphisms
\[
 \phi_i:D_i\longrightarrow V_i,
\]
for $i\in I$ where $D_i\subseteq L$ and $V_i\subseteq R$ are clopen and $\phi_i$ is referred to as a \textit{chart}. The edge space is identified with
$\bigsqcup_i\set i\times D_i$ via a homeomorphism $\Psi$, with endpoints $x$ and $\phi_i(x)$. For $A\subseteq L$, write $N(A)$
for its set of right neighbors, that is, $N(A) = \bigcup_{i \in I} \phi_i (A \cap D_i)$.

A matching $M$ is a subset of the edge set $E$ such that $s|_M$ and $r|_M$ are injective.  A matching $M$ is \emph{clopen} if it is clopen in $E$. It determines a partial
homeomorphism $M:\dom M\to\ran M$ with clopen domain and range.
An \textit{augmenting path} starts at an unmatched left vertex, ends at an unmatched
right vertex, and alternates between edges outside and inside the matching. This means a sequence of length $2k+1$ with vertices
\[x_0, y_0, x_1, y_1,\dots, x_k, y_k, \qquad x_i\in L, y_i\in R\]
such that 
\begin{enumerate}[label=(\roman*)]
    \item $x_0\notin \dom M$ and $y_k\notin \ran M$;
    \item there are edges $e_i\notin M$ for $0\leq i\leq k$ with $s(e_i)=x_i$ and $r(e_i)=y_i$;
\item there are edges $m_i\in M$ for $1\leq i\leq k$ with $s(m_i)=x_i$ and $r(m_i)=y_{i-1}$.
\end{enumerate}
We finally make the multigraph $(L, R, E)$ oriented with respect to a matching  $M$ in the following sense. We orient all edges from left to right and adjoin the
reverse of every edge in the matching $M$. It is straightforward to see a directed walk from an unmatched left vertex
to an unmatched right vertex can be shortened to an augmenting path. We refer to  \cite[Section 2.6]{GL} for more details on multigraphs and clopen matchings.

We use the following established augmentation lemma. It is stated here to
specify both the regularity and monotonicity that the proofs require.

\begin{lemma}[{\cite[Lemma 3.2]{GL}}]\label{lem:augmentation}
For a finite clopen presentation and a clopen matching $M_0$, there are
clopen matchings $M_h$ for $h \in \mathbb{N}$ satisfying
\[
 \dom M_h\subseteq\dom M_{h+1}\quad \text{ and } \quad
 \ran M_h\subseteq\ran M_{h+1} \quad \text{ for any } h \in \mathbb{N},
\]
and $M_h$ has no augmenting path of length at most $2h-1$ for $h\geq1$.
\end{lemma}

The lemma uses compact zero-dimensional spaces, with no metrizability
assumption.

\subsection{Graphs carried by compact open bisections}
In this subsection, we construct our multigraph. Let $L$ and $R$ be disjoint compact open subsets of
$\GU\times[m]$. We say that a finite clopen presentation is \emph{carried by compact open bisections in
$\CG$} if each chart $\phi_i$ is induced by a compact open bisection of
$\CG\times\mathcal R_m$ in the sense that there exists a compact open bisection $B_i=U_i\times\{(p_i, q_i)\}\subset \CG\times \CR_m$ such that $s(B_i)=D_i$, $r(B_i)=V_i$ and $\phi_i=r|_{B_i}\circ (s|_{B_i})^{-1}$. In this case, define a compact symmetric $K=\GU\cup\bigcup_{i\in I} (U_i\cup U^{-1}_i)$. Then the edges in the multigraph are implemented by elements in $K$.

\begin{theorem}\label{thm:hall}
Let $G$ be a fiberwise supramenable ample groupoid. Let a finite clopen presented bipartite graph $(L, R, E)$ be carried by compact open bisections in $\CG$. Suppose that
there is $\lambda>1$ such that
\begin{equation}\label{eq:strict-hall}
 \card{N(F)}\geq\lambda\card F
 \quad\text{for every nonempty finite }F\subseteq L.
\end{equation}
Then the graph has a clopen matching covering every left vertex.
\end{theorem}
\begin{proof}
Apply Lemma~\ref{lem:augmentation} to the empty matching and write
$U_h=L\setminus\dom M_h$ for $h\geq 0$. Note that $\{U_h: h\geq 0\}$ is a decreasing sequence. Thus, if no $M_h$ covers $L$, then the these compact
sets $U_h$ have a common point $v$.

Consider the graph component $\mathcal C$ containing $v$. Because the presentation of the multigraph is finite and hence all degrees are bounded, the component $\CC$ is countable. By compactness of $\set{0,1}^{E(\mathcal C)}$,  pass to a subsequence if necessary, one may assume 
matchings $1_{M_h|_\CC}$ converges on each edge of $\mathcal C$ to a matching on $\CC$, denoted by $M$.  The vertex $v$ is
unmatched because it was unmatched at every stage. Moreover $M$ has no
finite augmenting path. Indeed, any finite augmenting path $p$  and the unmatched status of
its endpoints depend on finitely many edge coordinates. Those coordinates
eventually agree with the subsequence, contradicting the absence of short
augmenting paths at sufficiently large stages.

Now, define 
\[T=\{u\in L\cap \CC: u\text{ is reachable from } v \text{ by a directed walk in } (L, R, E).  \}\]
Here a direct walk is a walk on the multigraph with respect to the orientation described above. Note that every vertex of $N(T)$ is matched. Otherwise, let $u\in T$ be such that $\phi_i(u)\in N(T)$ for some $i\in I$ is an unmatched vertex. By definition and deleting repeated vertices if necessary, there is a path $v=x_0, y_0, x_1, y_1,\dots, x_n=u, y_n=\phi_i(u)$, which is actually an augmented path. This is a contradiction and thus one has
\[
 M^{-1}(N(T))\subseteq T.
\]
For every finite nonempty $F\subseteq T$, the inverse matching is injective
on $N(F)$. Therefore, one has
\begin{equation}\label{eq:reachable-expansion}
 \card{M^{-1}(N(F))}=\card{N(F)}\geq\lambda\card F.
\end{equation}
 Write $v=(x,p)$ for a $x\in \GU$ and $p\in[m]$. We now lift the multigraph restricting to $\CC$ onto the fiber $\CG_x\times [m]$. By definition,  All underlying units from $\CG$ involved in the vertices of $\mathcal C$ lie in
$r(\CG\cdot x)$. Define a map:
\[\rho: \CG_x\times [m]\to r(\CG x)\times [m]\]
by $\rho(\gamma, j)=(r(\gamma), j)$. Then for any vertex $t=(y, j)\in T$, one has $\rho^{-1}(t)=\CG^y_x\times\{j\}$. 
Define $\widetilde{T}=\rho^{-1}(T)$. Suppose $e\in E$ is an edge in the multigraph with endpoints $t=(y, j)$ and $(z, k)$. Then because the multigraph $(L, R, E)$ is finite clopen presented, then there is a compact open bisection $U$ such that $g_e\in U$ with $s(g_e)=y$ and $r(g_e)=z$. In this case, we can define the lifted edge by
\[(e, \gamma): (\gamma, j)\to(g_e\gamma, k)\]
for any $\gamma\in \CG^y_x$. Then for the matching $M$, we lift it to
\[\widetilde{M}=\{(e, \gamma): e\in M, \gamma\in \CG^{r(\gamma)}_x\},\]
which is still a matching such that $\dom(\widetilde{M})=\rho^{-1}(\dom (M))$ and $\ran(\widetilde{M})=\rho^{-1}(\ran(M))$.

We then claim $|N(F)|\geq \lambda|F|$ for any finite $\emptyset\ne F\subset \widetilde{T}$. To show this, let $F\subset \widetilde{T}$ be non-empty and define a function $f(t)=|F\cap \rho^{-1}(t)|$. First observe $N(F)\subset \rho^{-1}(N(\rho(F)))$.
Note that for any right vertex $w\in N(\rho(F))$, and any neighborhood $t\sim w$, there are $f(t)$ elements in $F\cap \rho^{-1}(t)$ connecting to $f(t)$ distinct lifts of $w$.  This implies that
\[|N(F)\cap \rho^{-1}(w)|\geq \max_{t\sim w}f(t)\]
and therefore one has
\[|N(F)|\geq \sum_{w\in N(\rho(F))}\max_{t\sim w}f(t).\]
Then define $A_q=\{t\in T:f(t)\geq q\}$ for each integer $q\geq 1$. These are finite sets and only finitely many are non-empty. Then it is straightforward to see $w\in N(A_q)$ if and only if $\max_{t\sim w}f(t)\geq q$. This entails that 
\[\sum_{w\in N(\rho(F))}\max_{t\sim w}f(t)=\sum_{q\geq 1}|N(A_q)|.\]
Then by the assumption \eqref{eq:strict-hall}, one has
\[|N(F)|\geq \sum_{w\in N(\rho(F))}\max_{t\sim w}f(t)=\sum_{q\geq 1}|N(A_q)|\geq \sum_{q\geq 1}\lambda|A_q|=\lambda\sum_{t\in T}f(t)=\lambda |F|.\]

Finally, note that $N(F)\subset \rho^{-1}(N(\rho(F)))\subset \rho^{-1}(\ran(M))=\ran(\widetilde{M})$. This implies that 
 \[\card{\widetilde{M}^{-1}(N(F))}=\card{N(F)}\geq\lambda\card F.\]
one may identify $\widetilde{M}^{-1}(N(F))\subset N^{\widetilde{T}}_{K^2}(F)$ and thus
 \eqref{eq:reachable-expansion} implies
\[
 \card{N_{K^2}^{\widetilde T}(F)}\geq\lambda\card F.
\]
This contradicts Lemma~\ref{lem:labels} with $0<\varepsilon<\lambda-1$.
Hence some $U_h$ is empty, and the corresponding clopen matching covers $L$.
\end{proof}

\begin{lemma}\label{lem:witness-graph}
If $(n+1)a\leq nb$ in $S(\CG)$, there are disjoint compact open
representatives $A,B$ of $a,b$ and a finite clopen graph carried by compact open bisections of $\CG$, with
left vertex space $A$ and right vertex space $B$, such that
\[
 \deg(u)=n+1\ (u\in A),\qquad \deg(v)\leq n\ (v\in B).
\]
In particular, \eqref{eq:strict-hall} holds with $\lambda=(n+1)/n$.
\end{lemma}
\begin{proof}
Choose $m\in \N$ large enough and let $A, B\subset \GU\times [m]$ be disjoint compact open sets such that $[A]=a$ and $[B]=b$.
By \eqref{eq:subequiv}, the
inequality $(n+1)a\leq nb$ is witnessed by a finite clopen injection
\[
 J:A\times[n+1]\longrightarrow B\times[n]
\]
implemented piecewise by compact open bisections in $\CG\times \CR_k$ for a large integer $k$. This implies $n+1$ copies of $A$ can be moved into $n$ copies of $B$ with disjoint images. We now construct the multigraph as follows.

Define $L=A$, $R=B$ and $E=A\times [n+1]$ such that if $J(u, i)=(v, j)$ then the edge $(u, i)\in A\times [n+1]$ connects $u$ and $v$. Because $J$ is implemented by finitely many compact open bisections, it can be verified that the current $(L, R, E)$ has a finite clopen presentation and is carried by compact open bisections in $\CG$.

We finally calculate the degrees of vertices.
Let $u\in A$. Then the edge with $u$ as the left endpoint are exactly $(u,1),\dots, (u, n+1)$. This implies that $\deg(u)=n+1$. For $v\in B$, the edge with right endpoint $v$ come from $J^{-1}(\{v\}\times [n])$. Since $J$ is injective, one has $\deg(v)\leq n$.
Finally, for a finite $F\subseteq A$, exactly $(n+1)|F|$ many edges has left endpoints in $F$. The right endpoints of these edges are in $N(F)$, each of which has degree bounded by $n$. Therefore, one has
$(n+1)\card F\leq n\card{N(F)}$.
\end{proof}

Now we are ready to prove Theorem~\ref{thm:A}.
\begin{proof}[Proof of Theorem~\ref{thm:A}]
Given $(n+1)a\leq nb$ and write $[A]=a$ and $[B]=b$ with $A, B$ are disjoint compact open sets in $\GU\times [m]$ for some $m$. It follows from Lemma~\ref{lem:witness-graph} 
and
Theorem~\ref{thm:hall} to obtain a clopen matching $M$ such that $A\subset\dom(M)$. Since the matching $M$ is implemented by finitely many  compact open bisections, one has $a\leq b$. Then if $\CG$ is minimal and $\sigma$-compact, then it follows from \cite[Proposition 2.10]{ABBL} that $\CG$ has groupoid strict comparison.
\end{proof}

The following is a direct corollary of Proposition~\ref{prop:growth} and Theorem~\ref{thm:A}.

\begin{cor}\label{cor:growth-au}
Let $\CG$ be an ample groupoid on a compact unit space. If $\CG$ has pointwise
local subexponential growth, then $S(\CG)$ is almost unperforated.
\end{cor}

\section{Topological amenability and finite-order bisections}\label{sec:top}

In this section we  set $P=\langle\alpha\mid\alpha^3=e\rangle$,
$Q=\langle\beta\mid\beta^2=e\rangle$, and $\Pi=P*Q$.
We use two consequences of the matching results in \cite{GL}.

\begin{lemma}[{\cite[Lemma 4.1]{GL}}]\label{lem:gl-amenable}
If $\Pi\curvearrowright Z$ is a topologically amenable action on a compact
zero-dimensional space and $P$ acts freely, then a clopen $C_Z\subseteq Z$
exists with
\begin{equation}\label{eq: CZ property}
     Z=C_Z\sqcup\alpha C_Z\sqcup\alpha^2C_Z,
 \qquad C_Z\cap\beta C_Z=\varnothing.
\end{equation}
\end{lemma}

The following is precisely the combination of \cite[Lemma 3.4 and Proposition 3.3]{GL}.

\begin{lemma}\label{lem:gl-completion}
Let $\Pi$ act on a compact zero-dimensional space $Y$. Suppose that $F$ is
clopen and $P$-invariant, that $P$ acts freely on $F$ and fixes $Y\setminus F$
pointwise, and that $Q$ acts freely on $Y$. In the incidence graph
\[
 L=F/P,\qquad R=Y/Q,\qquad E=F,
\]
which is defined by $s(y)=Py$ and $r(y)=Qy$ for any $y\in F$,
let $M_0$ be a clopen matching. If each unmatched left vertex has a
representative fixed by a nonidentity element of $\Pi$, then there is a
clopen matching covering $L$.
\end{lemma}

\subsection{The groupoid on the set with trivial pointwise kernel}

For a compact open full bisection $p$ over $Y$, write $\theta_p$ for its
induced homeomorphism, i.e the corresponding element in the topological full group $[[\CG|_Y]]$. 

\begin{prop}\label{prop:bisection-selection}
Let $\mathcal G$ be a topologically amenable ample Hausdorff
groupoid, and let $Y\subseteq\mathcal G^{(0)}$ be compact open. Suppose $p,q$
are compact open full bisections of $\mathcal G|_Y$ satisfying
\[
 p^3=1_Y,\qquad q^2=1_Y.
\]
Let $F\subseteq Y$ be clopen and $\theta_p$-invariant. Suppose every
$\theta_p$-orbit in $F$ has three points, $p|_{Y\setminus F}=1_{Y\setminus F}$,
and $\theta_q$ has no fixed point on $Y$. Then there is a clopen $C\subseteq F$ with
\[F=C\sqcup\theta_p(C)\sqcup\theta_p^2(C),\qquad
 C\cap\theta_q(C)=\varnothing.
\]
\end{prop}
\begin{proof}
The bisection identity $p^3=q^2=1_Y$ yields a homomorphism
$\varphi: w\mapsto B_w$ from $\Pi=P*Q=\langle \alpha\rangle*\langle \beta\rangle$ to the topological full group $[[\CG\mid_Y]]$, with
$B_\alpha=\theta_p$ and $B_\beta=\theta_q$. Set
\[
 H=\bigcup_{w\in\Pi}B_w.
\]
Since $B_uB_v=B_{uv}$ and $B_w^{-1}=B_{w^{-1}}$, this is a subgroupoid
with unit space $Y$. It is open in $\mathcal G|_Y$ and hence topologically
amenable. Let $b_w(y)\in B_w$ be the unique element whose source is $y$. Note that the homomorphism $\varphi$ yields an action of $\Pi$ on $Y$ by $w\cdot y=r(b_w(y))$. Then the following cocycle equality holds.
\begin{equation}\label{eq:bisection-cocycle}
 b_{uv}(y)=b_u(v\cdot y)b_v(y).
\end{equation}

Define
\[
 \Lambda_y=\set{w\in\Pi:b_w(y)=y},\qquad
 Z=\set{y\in Y:\Lambda_y=\set e}.
\]
Note that $b_w(y)=y$ means $y\in B_w$, which implies $w\cdot y=y$ and $y\in B_{w^{-1}}$. Combining with  the cocycle equality \eqref{eq:bisection-cocycle}, one has $\Lambda_y$ is a subgroup. For $w\ne e$, the set
$\{y\in Y:b_w(y)=y\}=B_w\cap H^{(0)}$ is clopen. Consequently, 
\[Z=Y\setminus \bigcup_{w\ne e} \{y\in Y: b_w(y)=y\}\]
is closed. In addition, 
formula~\eqref{eq:bisection-cocycle} shows
$\Lambda_{v\cdot y}=v\Lambda_yv^{-1}$, from which one obtains that $Z$ is $\Pi$-invariant and thus
$H$-invariant. For any $y\notin F$. Since $\theta_p$ is the identity map on $Y\setminus F$, one has $b_\alpha(y)=y$, which implies that $y\notin Z$ and thus $Z\subseteq F$.

We now define an isomorphism of topological groupoids $ \Phi:Z\rtimes \Pi\longrightarrow H|_Z$ defined by
\begin{equation}\label{eq:kernel-isomorphism}
 \Phi(w\cdot z,w,z)=b_w(z).
\end{equation}
The cocycle identity \ref{eq:bisection-cocycle} shows that $\Phi$ preserves multiplicativity. Surjectivity of $\Phi$ follows from the
definition of $H=\bigcup_{w\in \Pi}B_w$. To show the injectivity of $\Phi$, suppose $b_u(z)=b_v(z)$. Note that $b_v(z)^{-1}=b_{v^{-1}}(v\cdot z)$ by definition. This then implies that
$b_{v^{-1}u}(z)=b_v(z)^{-1}b_u(z)=z$ and thus $u=v$ by the definition of
$Z$. Each slice $(Z\times\set w\times Z)\cap (Z\rtimes \Pi)$ maps homeomorphically onto
$B_w\cap (H|_Z)$, an open bisection of the reduction. Hence the bijection $\Phi$ is
a homeomorphism.

The invariant closed reduction $H|_Z$ is amenable. By
\eqref{eq:kernel-isomorphism}, the action $\Pi\curvearrowright Z$ is
topologically amenable. The $P$-action on $Z$ is free because $Z\subseteq F$.
Apply Lemma~\ref{lem:gl-amenable}, to obtain a relatively clopen $C_Z$ in $Z$ satisfying the property \eqref{eq: CZ property} in  Lemma~\ref{lem:gl-amenable}. If $Z$ is empty, simply take
$C_Z=\varnothing$.

Extend $C_Z$ to a clopen subset $C'\subseteq F$ such that $C'\cap Z=C_Z$. Such an extension exists
compact sets $C_Z$ and $Z\setminus C_Z$ are thus closed in the clopen set $F$. 
Define
\[
 C_0=C'\setminus\bigl(\alpha(C')\cup\alpha^2(C')\cup\beta_q(C')\bigr).
\]
Then $C_0$ meets every $P$-orbit and every $Q$-orbit in at most one point because by definition,
\[C_0\cap \alpha(C_0)=C_0\cap \alpha^2(C_0)=C_0\cap \beta(C_0)=\emptyset.\]
In addition, one still has $C_0\cap Z=C_Z$. Indeed, since $Z$ is $\Pi$-invariant, one has $\alpha(C')\cap Z=\alpha(C_Z)$, $\alpha^2(C')\cap Z= \alpha^2(C_Z)$, and $\beta(C')\cap Z=\beta(C_Z)$. Then properties \eqref{eq: CZ property} for $C_Z$ shows that $C_0\cap Z=C'\cap Z=C_Z$.
Consequently $C_0$
is a clopen partial matching in the incidence graph $F/P$--$Y/Q$ that covers
every left vertex $Pz$ represented by a $z\in Z$.

Any unmatched left vertex has a representative $y\notin Z$. There is
a $w\ne e$ with $b_w(y)=y$, and hence $wy=y$. It then follows from Lemma~\ref{lem:gl-completion} that there exists a clopen matching covering $L=F/P$. 
Its clopen edge set $C\subseteq F$ meets
each $P$-orbit once and each $Q$-orbit at most once, as required. Equivalently, one has $F=C\sqcup \alpha(C)\sqcup \alpha^2(C)$ and $C\cap \beta(C)=\emptyset$.
\end{proof}

Now we are ready to prove Theorem~\ref{thm:B}.
\begin{proof}[Proof of Theorem~\ref{thm:B}]
Suppose $3a\leq2b$. If $a=0$, there is nothing to prove. Otherwise choose
representatives and a sufficiently large finite amplification
$\mathcal G=G\times\mathcal R_d$ in which the inequality is witnessed.
Let $B_0,B_1$ be disjoint copies of a representative of $b$, define
$Y=B_0\sqcup B_1$. In addition, because $3a\leq 2b$, without loss of  generality, one may assume there exists disjoint $A_0,A_1,A_2\subseteq Y$ as three copies of a representative of $a$. Put
$F=A_0\sqcup A_1\sqcup A_2$.

Since $[A_0]=[A_1]=[A_2]=a$, there are compact open bisections $t_0, t_1$ with $s(t_0)=A_0$, $r(t_0)=A_1$ and $s(t_1)=A_1$, $r(t_1)=A_2$.
 Set
$t_2=(t_1t_0)^{-1}$ and define
\[
 p=t_0\sqcup t_1\sqcup t_2\sqcup1_{Y\setminus F}.
\]
Similarly, there exists compact open bisection $\tau$ with $s(\tau)=B_0$ and $r(\tau)=B_1$. Define $q=\tau\sqcup\tau^{-1}$ and thus
$q^2=1_Y$.

The groupoid $\mathcal G$ is topologically amenable. It follows from
Proposition~\ref{prop:bisection-selection} to obtain a clopen set $C\subseteq F$ such that
\[C\sqcup \theta_p(C)\sqcup \theta^2_p(C)=F=A_0\sqcup A_1\sqcup A_2.\]
and $C\cap\theta_q(C)=\varnothing$. We claim that $[C]=a$. Indeed, define $C_i=C\cap A_i$ for $i=0,1,2$ and therefore $C=C_0\sqcup C_1\sqcup C_2$. For any $x\in A_0$, there is only one point among  $x, \theta_p(x), \theta^2_p(x)$ belonging to $C$. Let $\theta_p^i(x)\in C$ for some $i=0, 1,2$. Then by our definition, note that $\theta_p^i(x)\in A_i$ as well. Therefore, one has $x\in (\theta_p^i)^{-1}(C_i)$. This implies that
\[A_0=C_0\sqcup \theta_p^{-1}(C_1)\sqcup\theta_p^{-2}(C_2)\]
and thus 
\[[C]=[C_0\sqcup C_1\sqcup C_2]=[A_0]=a.\]

Then define $D=(C\cap B_0)\sqcup \theta_q(C\cap B_1)$, which is a subset of $B_0$. Note that this map keeps $C\cap B_0$ fixed and move $C\cap B_1$ into $B_0$ using $q$. Therefore, one has 
\[a=[C]=[D]\leq [B_0]=b.\]

If $\CG$ is minimal, then $S(\CG)$ is a simple conical refinement monoid by Remark \ref{rmk: type semigroup conical and simple}. Then it follows from Lemma~\ref{lem:three-algebra} that $S(\CG)$ is almost unperforated. Then if $\CG$ is  $\sigma$-compact, then it follows from \cite[Proposition 2.10]{ABBL} that $\CG$ has groupoid strict comparison.
\end{proof}

\section{Applications}\label{sec:app}
In this section, we record several applications of Theorem \ref{thm:A} and \ref{thm:B}. Matui's AH conjecture (see \cite[Conjecture 2.9]{Matui}) propose a natural connection between the first two homology group of a minimal ample groupoid and the abelianisation of its topological full groups. It follows from \cite[Corollary E]{Li} that the AH conjecture holds for any minimal ample groupoid $\CG$ satisfying groupoid strict comparison and containing no isolated points in the unit space $\GU$. Therefore, we have the following application in a straightforward way. Note that if $\CG$ is minimal and $\GU$ is compact and infinite, then necessarily $\CU$ is \textit{perfect} in the sense that there is no isolated points. See, e.g., \cite[Remark 8.10]{MW2}.  

\begin{thm}\label{thm: AH}
    Let $\CG$ be an ample minimal groupoid on an infinite compact unit space such that $\CG$ is fiberwise supramenable or $\CG$ is topological amenable. Then $\CG$ satisfies the AH conjecture.
\end{thm}

Another application is on regularity properties of groupoids and groupoid $C^*$-algebras. It was introduced by the second and third author in \cite{MW2} a new regularity property for a locally compact  Hausdorff groupoid $\CG$ with compact unit space called \textit{almost elementariess} and it was shown in \cite{MW2} that this propety implies that the reduced groupoid $C^*$-algebra $C^*_r(\CG)$ is tracially $\CZ$-stable. Then Theorem \ref{thm:A} and \ref{thm:B} shows that in many cases, almost elementariness is actually equivalent to the weaker condition called \textit{almost elementariness in measure}, which is easier to verify. See, e.g., \cite{M}. We refer to \cite[Section 7]{MW2} for definitions of almost elementariness and almost elementariness in measure. 

\begin{thm}\label{thm:AE}
    Let $\CG$ be an ample minimal groupoid on a compact unit space such that $\CG$ is fiberwise supramenable or $\CG$ is topological amenable. Then $\CG$ is almost elementary in measure if and only if $\CG$ is almost elementary.  
\end{thm}
\begin{proof}
    This is a combination of Theorem \ref{thm:B} and \cite[Theorem 8.13]{MW2}.
\end{proof}

The final application is on the pure infiniteness of topological amenable groupoids and their reduced groupoid $C^*$-algebras.

\begin{thm}\label{thm: pure inf}
    Let $\CG$ be an ample minimal topological amenable groupoid on an infinite compact unit space. 
    \begin{enumerate}
        \item If there is no $\CG$-invariant probability measure on $\GU$, then $\CG$ is purely infinite.  
        \item If, in addition, $\CG$ is topological principal, then its reduced groupoid $C^*$-algebra $C^*_r(\CG)$ is unital, simple,  nuclear and purely infinite $C^*$-algebra.
        \item If $\CG$ is further assumed to be second countable, then $C^*_r(\CG)$ is a unital Kirchberg algebra satisfying the UCT.
    \end{enumerate}
\end{thm}
\begin{proof}
    Theorem \ref{thm:B} implies that the $\CG$ is purely infinite by \cite[Theorem C]{Mpure}. If $\CG$ is topological principal, then it follows from \cite[Corollary 1.1]{Mpure} that $C^*_r(\CG)$ is unital, simple, nuclear purely infinite $C^*$-algebra. If $\CG$ is second countable, then by definition, it is a unital Kirchberg algebra satisfying the UCT.
\end{proof}

We finally remark that pure infiniteness is equivalent to almost elementariness of $\CG$ whenever $\CG$ is minimal and topological principal. Therefore, as an alternative approach, it follows from \cite[Theorem 13.11]{MW2} that $C^*_r(\CG)$ is $\CZ$-stable.

\section{Acknowledgement}
The original idea of fiberwise supramenability for groupoids came to the second author during his attendance in the workshop ``Cuntz semigroup and Dynamics'' in G\"{o}teborg, Sweden. The second author would like to thank the organizers for the invitation and the kind host for this stay. 

During the preparation of the manuscript, GPT 6 Astra was used for searching necessary references, examining mathematical arguments, and language editing.  The authors have checked and substantially revised all mathematical arguments generated by AI and take full responsibility for the final version.

\end{document}